\documentclass{article}
\usepackage{amssymb,amsmath, amsthm}
\usepackage{bm}
\newtheorem{theo}{Theorem}[section]
\newtheorem{lem}[theo]{Lemma}

\newtheorem{prop}[theo]{Proposition}

\begin{document}

\title{On Mobius and Liouville functions of order $k$ }
\author{Yusuke Fujisawa}

\maketitle
\begin{abstract}
Let $F$ be a number field, $k$ a positive integer. 
In this paper, 
we define the Mobius and Liouville functions of order $k$ in $F$. 
We give a formula about the partial sums of them by using elementary number theory 
and complex analysis.  
Moreover, we also consider the number of 
$k$-free ideals of the integer ring of $F$. 
\end{abstract}

\section{Introduction}

For a large number $x>0$,  there are classical statements that 
$\sum_{n \leq x}\mu(n)=O(x \exp(-c\sqrt{\log x}))$ and 
$\sum_{n \leq x}\lambda(n)=O(x \exp(-c\sqrt{\log x}))$ 
where  $\mu$ (resp. $\lambda$) is the Mobius (resp. Liouville) function. 
Here and throughout 
$c$  denotes a positive constant not necessarily the same at different occurrences. 
 It is also well-known that the number of square-free positive integers not 
 exceeding $x$ is  $\frac{6}{\pi^2}x+O(\sqrt{x})$ (see e.g. Montgomery and Vaughan \cite{montgomery} ).
Our purpose is to generalize these classical results.

Let $F$ be a number field of degree $d$, $O_F$ the integer ring of $F$, 
$I_F$ the set of all non-zero ideals of $O_F$. 
Throughout this paper, $A, B, C, D, E$ are non-zero ideals of $O_F$ and $P$ is a prime ideal 
of $O_F$. 

Let $k$ be a positive integer. 
Now, we introduce the Mobius and  Liouville functions of order $k$ in $F$. 
The function $\mu_{k, F}:I_f \rightarrow \{ 0 \pm 1 \}$ is defined as follows. 
For a prime ideal $P$ and a non-negative integer $m$, put
\begin{align*}
\mu_{k, F}(P^m)=
\begin{cases}
1 \;\;\;\;\; &\textrm{if} \; m < k,  \\
-1 \;\;\;\;\; &\textrm{if} \; m = k, \\
0 \;\;\;\;\;\; &\textrm{if} \; k< m.
\end{cases}
\end{align*}
For $A=P_1^{e_1}\cdots P_g^{e_g} \in I_F$ where $P_1$, $\cdots$, $P_g$ 
are distinct prime ideals and $e_1$, $\cdots$, $e_g$ are positive integers, 
we put 
\begin{align*}
\mu_{k, F}(A)=\prod_{i=1}^{g}\mu_k(P_i^{e_i}). 
\end{align*}

The function $\lambda_{k, F}: I_F \rightarrow \{ 0, \pm 1 \}$ is defined as follows. 
For a prime ideal $P$ and a non-negative integer $m$, put
\begin{align*}
\lambda_{k, F}(P^m)=
\begin{cases}
1 \;\;\;\;\; &\textrm{if} \; m \equiv 0 \mod k+1,  \\
-1 \;\;\;\;\; &\textrm{if} \; m \equiv 1 \mod k+1, \\
0 \;\;\;\;\;\; &\textrm{otherwise}.
\end{cases}
\end{align*}
For $A=P_1^{e_1}\cdots P_g^{e_g} \in I_F$ where $P_1$, $\cdots$, $P_g$ 
are distinct prime ideals and $e_1$, $\cdots$, $e_g$ are positive integers, 
we put 
\begin{align*}
\lambda_{k, F}(A)=\prod_{i=1}^{g}\lambda_k(P_i^{e_i}). 
\end{align*}

For simplicity, we denote $\mu_{k, F}$(resp. $\lambda_{k, F}$) by 
$\mu_k$(resp. $\lambda_k$) since we fix a number field $F$. 
Note that 
$\mu_{1}$(resp. $\lambda_1$) is the ordinary Mobius (resp. Lioville) function.

For a large real number $x>0$, we put 
\begin{align*}
M_k(x) :=\sum_{N(A) \leq x}\mu_k(A), \\
L_k(x) :=\sum_{N(A) \leq x}\lambda_k(A), 
\end{align*}
and 
\begin{align*}
Q_k(x) :=\sum_{N(A) \leq x}|\mu_{k-1}(A)|
\end{align*}
where $N(A)$ is the norm of $A$, i. e., $N(A)=[O_F:A]$.

For an ideal $A \in I_F$, we say $A$ {\it $k$-free} if there is no prime ideal $P$ such that $P^k|A$. 
By the definition, $|\mu_{k-1}|$ is the characteristic function of $k$-free ideals, and $Q_k(x)$ is the number of $k$-free ideals of $O_F$ whose norm is less than or equal to $x$.

%Let $q_k$ be the characteristic function of $Q_k$, that is, 
%$q_k(A)=1$ if $A \in Q_k$ and $q_k(A)=0$ otherwise. 
%By the definition, we see $q_{k}(A)=|\mu_{k-1}(A)|$. 
%In this paper, we consider  and 
%$Q_k(x) :=\sum_{N(A) \leq x}|\mu_{k-1}(A)|$ 
%where  

\begin{theo}\label{mobius}
For a integer $k \geq 2$, we have 
\begin{align*}
M_k(x)=
\frac{c_F}{\zeta_F(k)}x \sum_{A \in I_F}
\frac{\mu_1(A)J_1(A)}{J_k(A)N(A)}
+ O \left( x^\frac{1}{k} \log x\right)
\end{align*}
where $J_k(A)=N(A)^k \prod_{P|A}\left( 1-N(P)^{-k} \right)$, 
$\zeta_F(s)$ is the Dedekind zeta function of $F$, and $c_F$ is the residue of 
$\zeta_F(s)$ at $s=1$.
\end{theo}

An analogue of the above theorem is shown by Apostol \cite{apostol} when $F=\mathbb{Q}$.

\begin{theo}\label{liouville} Let $k$ be a positive integer. 
There exist  positive constants $c$ and $c'$ such that
\begin{align*}
L_k(x)\ll \zeta((k+1)(1-c/\sqrt{\log x} ))x\exp(-c'\sqrt{\log x}).
\end{align*} 
In addition, 
it holds that 
\begin{align*}
L_k(x)=O(x^{\frac{1}{2}+\varepsilon})
\end{align*}
 for any $\varepsilon >0$ 
if and only if the grand Riemann Hypothesis (GRH) holds for $\zeta_F(s)$. 
\end{theo}

Since $1<(k+1)(1-c/\sqrt{\log x})$ for $k \geq 1$ and large $x>0$, we see 
$\zeta((k+1)(1-c/\sqrt{\log x} ))=O(1)$. Thus, we can roughly write 
$L_k(x)=O(x\exp(-c\sqrt{\log x}))$.
The author and Minamide \cite{minamide} 
showed that 
$M_1(x)=O(x\exp(-c\sqrt{\log x}))$ and 
$L_1(x)=O(x\exp(-c\sqrt{\log x}))$ for arbitrary number field.

In addition, the number of $k$-free ideals of $O_F$ is estimated as 
the next theorem.

\begin{theo}\label{qk}
Let $k \geq 2$ and $d=[F:\mathbb{Q}]>1$. Then, 
we have  
\begin{align*}
Q_k(x)=\frac{c_F}{\zeta_F(k)}x+
\begin{cases}
O(x^{\frac{1}{2}}) \;\;\; \textrm{if $(d,k)=(2, 2)$}, \\
O(x^{\frac{1}{3}}\log x) \;\;\; \textrm{if $(d, k)=(2, 3)$}, \\
O(x^{\frac{1}{2}}\log x) \;\;\; \textrm{if $(d, k)=(3, 2)$}, \\
O(x^{\frac{d-1}{d+1}}) \;\;\; \textrm{otherwise.}
\end{cases}
\end{align*}
\end{theo}

The number of $k$-free positive integers was considered in \cite{cohen}, 
\cite{walfisz}, etc.

%{\bf Acknowledgements.}
%The author is grateful to Prof. Yoshio Tanigawa, Prof. Kalyan Chakraborty and 
%Prof. Makoto Minamide for encouragements. 

\section{Preliminary and Background} 

First, we give an important lemma. 
For a large real number $x>0$, 
we denote the number of non-zero ideals of $O_F$ 
whose norm is less than or equal to $x$ by $[x]_{F}$. 

\begin{lem}(See Berndt \cite{berndt}, Murty and Order \cite{murty} and other papers cited there.)
For a large real number $x>0$, we have 
\begin{align*}
[x]_F=c_Fx+O(x^{\frac{d-1}{d+1}})
\end{align*} 
where the constant $c_F$ is the residue of the Dedekind zeta function $\zeta_F(s)$ 
of $F$ at $s=1$. 
\end{lem}

Next, we review some fundamental facts on arithmetical functions. 
In this note, an arithmetical function is a function from $I_F$ to $\mathbb{C}$.
For arithmetical functions $f$ and $g$, we define the Dirichlet convolution 
$f*g$ of them as 
\begin{align*}
f*g(A)=\sum_{D|A}f(D)g\left( \frac{A}{D} \right).
\end{align*}
This convolution is associative and  commutative.
The function $\delta$ such that $\delta(1)=1$ and $\delta(A)=0$ for $A \neq 1$ is 
the identity element of this convolution. For an arithmetical function $f$, 
there is the inverse of $f$ if and only if $f(1) \neq 0$. 
An arithmetical function $f$ is called multiplicative if $f(1)=1$ and $f(AB)=f(A)f(B)$ whenever $(A, B)=1$. 
For example,  $\mu_k$, $\lambda_k$  and  $\delta$ is multiplicative.
The Dirichlet convolution of two multiplicative function is also multiplicative 
and the Dirichlet inverse of multiplicative function is also multiplicative.

Now, we describe some properties of M{\"o}bius and Liouville functions of order $k$. 
We shall show some lemmas.

\begin{lem}\label{k-freelem} 

For a positive integer $k$, we have
\begin{align*}
|\mu_k(A)|=\sum_{D^{k+1}| A}\mu_1(D).
\end{align*}
\end{lem}
\begin{proof}
For a prime ideal $P$ and a positive integer $m$, 
we see 
\begin{align*}
|\mu_k(P^m)|=\sum_{D^{k+1}| P^m}\mu_1(D).
\end{align*} 
Since they are multiplicative functions, 
the assertion holds.  
\end{proof}

The next lemma plays an important role to prove Theorem \ref{mobius}.  
Let $\sigma_s(A)=\sum_{D|A}N(D)^s$ for $s \in \mathbb{C}$.

\begin{lem}\label{functionF}
Let $\alpha$ be a constant such that $(d-1)/(d+1) \leq \alpha < 1$. 
For $k\geq 2$ and $A \in I_F$, put
\begin{align*}
G_A(x)=\sum_{N(B) \leq x}\mu_{k-1}(B)\mu_{k-1}(A^{k-1}B).
\end{align*} 
Then,  
\begin{align*}
G_A(x)=\frac{c_F\mu_1(A)J_1(A)N(A)^{k-1}}{\zeta_F(k)J_k(A)}x 
+ O(x^{\frac{1}{k}}\sigma_{-\alpha}(A))
\end{align*}
unless $k\alpha = 1$.
\end{lem}
\begin{proof}
This is shown by a similar argument in \cite{apostol}. 
If $(A, B) \neq 1$, then $\mu_{k-1}(A^{k-1}B)=0$. 
If $(A, B) =1$, we see 
\begin{align*}
\mu_{k-1}(B)\mu_{k-1}(A^{k-1}B)=|\mu_{k-1}(B)|\mu_1(A)
\end{align*}
by the property $\mu_k(A^k)=\mu_1(A)$. 
Hence, 
\begin{align*}
G_A(x)=\mu_1(A)\sum_{N(B) \leq x \atop (A, B)=1}|\mu_{k-1}(B)|. 
\end{align*} 
By Lemma \ref{k-freelem} and the formula 
\begin{align*}
\sum_{N(C) \leq X \atop (C, A)=1} 1 = 
\sum_{E|A}\mu_1(E) \left[ \frac{X}{N(E)} \right], 
\end{align*} 
we have 
\begin{align*}
G_A(x) &=\mu_1(A)\sum_{N(B) \leq x \atop (A, B)=1}\sum_{D^k | B}\mu_1(D) \\
&= \mu_1(A) \sum_{N(D^k) \leq x \atop (A, D)=1}\mu_1(D)
\sum_{N(C) \leq x/N(D^k) \atop (C, A)=1}1 \\
&= \mu_1(A) \sum_{N(D^k) \leq x \atop (A, D)=1}\mu_1(D) 
\sum_{E| A}\mu_1(E)\left[ \frac{x}{N(ED^k)} \right]_F. 
\end{align*}
By the estimate $[x]_F=c_Fx+O(x^{\alpha})$, the above expression is 
\begin{align}\label{daiji}
c_Fx\mu_1(A)\sum_{E|A}\frac{\mu_1(E)}{N(E)}
\sum_{N(D^k) \leq x \atop (A, D)=1}\frac{\mu_1(D)}{N(D^k)}
+O\left( \sum_{E|A}\sum_{N(D^k) \leq x \atop (A, D)=1}
 \frac{x^\alpha}{N(ED^k)^\alpha}  \right). 
\end{align}

Since   
\begin{align*}
\sum_{N(D^k) \leq x \atop (A, D)=1}\frac{\mu_1(D)}{N(D^k)}
&= \sum_{(A, D)=1}\frac{\mu_1(D)}{N(D^k)} 
+ O\left( \sum_{N(D)>\sqrt[k]{x}} \frac{1}{N(D)^k} \right) \\
&= \frac{1}{\zeta_F(k)}\prod_{P|A}\left( 1-N(P)^{-k} \right)^{-1} + 
O\left( x^\frac{1-k}{k} \right) \\
&= \frac{N(A)^k}{\zeta_F(k)J_k(A)}+O(x^\frac{1-k}{k})
\end{align*}
and 
$\sum_{E|A}\mu_1(E)/N(E)=J_1(A)/N(A)$, 
the first term in (\ref{daiji}) is 
\begin{align*}
c_Fx\frac{\mu_1(A)J_1(A)N(A)^{k-1}}{\zeta_F(k)J_k(A)}+O(x^{\frac{1}{k}}).
\end{align*}
Next, we consider the O-term in (\ref{daiji}). We see that
\begin{align*}
x^\alpha \sum_{E|A}\frac{1}{N(E)^\alpha}\sum_{N(D) \leq \sqrt[k]{x} \atop (A, D)=1}
 \frac{1}{N(D)^{k\alpha}} 
&< 
 x^\alpha \sigma_{-\alpha}(A)
 \sum_{N(D) \leq \sqrt[k]{x}}
 \frac{1}{N(D)^{k\alpha}}  
\\
&\ll \begin{cases}
x^{\frac{1}{k}}\sigma_{-\alpha}(A) \;\;\; \textrm{if $k\alpha \neq 1$,} \\
x^\alpha \sigma_{-\alpha}(A) \log x^{\frac{1}{k}} \;\;\; 
\textrm{if $k\alpha=1$.}
\end{cases}
 \end{align*}

\end{proof}

Suppose $k \geq 2$. Next, we introduce the function $q_k$ in order to 
get the generating function of $\lambda_k$.
We  define the arithmetical function $q_k$ as 
\begin{align*}
q_k(A)
=
\begin{cases}
1 \;\;\;\;\;\; \textrm{$A$ is $k$-free}  \\
0 \;\;\;\;\;\;  \textrm{otherwise},
\end{cases}
\end{align*}
Namely, $q_k=|\mu_{k-1}|$, that is, the characteristic function of 
the set of all $k$-free ideals of $O_F$. 
We see the function $q_k$  is the 
Dirichlet inverse of $\lambda_{k-1}$. 

\begin{lem}\label{hodai1} 
For $k \geq 2$,  we have $q_k*\lambda_{k-1}=\delta$. 
\end{lem}
\begin{proof}
For a prime ideal $P$ and a positive integer $m$, we see 
\begin{align*}
q_k*\lambda_{k-1}(P^m)=
\begin{cases}
1 \;\;\;\;\;\; \textrm{$m=0$,} \\
0 \;\;\;\;\;\; \textrm{otherwise}.
\end{cases}
\end{align*}
Since $q_k$ and $\lambda_{k-1}$ are multiplicative, 
$q_k * \lambda_{k-1}$ is multiplicative. 
Hence, the assertion holds.
\end{proof}

Now, we have the generating function of $\lambda_{k-1}$.

\begin{lem}\label{hodai2} Let $k \geq 2$. For $\rm{Re}\;s >1$, we have 
\begin{align*}
\frac{\zeta_F(ks)}{\zeta_F(s)}=\sum_{A \in I_F}\frac{\lambda_{k-1}(A)}{N(A)^s}
\end{align*}
where $\zeta_F(s)$ is the Dedekind zeta function of $F$. 
\end{lem}
\begin{proof} 
Note that for arithmetical functions $f$ and $g$ we formally see 
\begin{align*}
\left( \sum_{A \in I_F}\frac{f(A)}{N(A)^s} \right) 
\left( \sum_{A \in I_F}\frac{g(A)}{N(A)^s} \right) 
=
\left( \sum_{A \in I_F}\frac{f*g(A)}{N(A)^s} \right).
\end{align*} 
Thus, 
\begin{align*}
\left( \sum_{A \in I_F}\frac{\lambda_{k-1}(A)}{N(A)^s} \right) 
\left( \sum_{A \in I_F}\frac{q_k(A)}{N(A)^s} \right) 
=1
\end{align*} 
from Lemma \ref{hodai1}.
By the Euler product of the Dedekind zeta function, 
we see 
\begin{align*}
\frac{\zeta_F(s)}{\zeta_F(ks)} 
= \prod_P 
\left( 1+\frac{1}{N(P)^s}+\cdots +\frac{1}{N(P)^{(k-1)s}} \right) = 
\sum_{A \in I_F}\frac{q_k(A)}{N(A)^s}.
\end{align*} 
Therefore, 
\begin{align*}
 \sum_{A \in I_F}\frac{\lambda_{k-1}(A)}{N(A)^s} = 
 \frac{\zeta_F(ks)}{\zeta_F(s)}. 
\end{align*} 
\end{proof}

\section{Proof of Theorem \ref{mobius}}
We can check   
$\mu_k(P^m)=\sum_{D^k|P^m}\mu_{k-1}( P^m/D^{k})\mu_{k-1}( P^m/D)$ 
for a prime ideal $P$ and a integer $m>0$. 
Since the both sides are multiplicative, we have  
\begin{align*}
\mu_k(A)=\sum_{D^k| A}
\mu_{k-1}\left( \frac{A}{D^{k}} \right)\mu_{k-1}\left( \frac{A}{D} \right)  
\end{align*}
 for any $A \in I_F$. 
 By this expression, 
\begin{align*}
M_k(x) &= \sum_{N(A)\leq x}\sum_{D^k|A}
\mu_{k-1}\left(\frac{A}{D^k}\right)
\mu_{k-1}\left(\frac{A}{D} \right) \\
&= \sum_{N(D)^k \leq x}\sum_{N(B) \leq x/N(D)^k}\mu_{k-1}(B)\mu_{k-1}(D^{k-1}B) \\
&= \sum_{N(D)^k \leq x}G_D\left( \frac{x}{N(D)^k} \right).
\end{align*}
If $k\alpha \neq 1$, 
\begin{align*}
G_{A}\left( \frac{x}{N(A)^k} \right)
=
\frac{c_F\mu_1(A)J_1(A)}{\zeta_F(k)J_k(A)N(A)} x 
+ O\left(x^\frac{1}{k}\frac{\sigma_{-\alpha}(A)}{N(A)}\right)
\end{align*}
by Lemma \ref{functionF}. Thus, 
\begin{align}\label{daiji2}
M_k(x)=\frac{c_Fx}{\zeta_F(k)}
\sum_{N(A) \leq \sqrt[k]{x}}
\frac{\mu_1(A)J_1(A)}{J_k(A)N(A)} 
+ 
O\left( 
x^{\frac{1}{k}}
\sum_{N(A) \leq \sqrt[k]{x}}
\frac{\sigma_{-\alpha}(A)}{N(A)} \right).
\end{align}
Since 
\begin{align*}
\sum_{N(A) \leq \sqrt[k]{x}}
\frac{\mu_1(A)J_1(A)}{J_k(A)N(A)}
= \sum_{A \in I_F}
\frac{\mu_1(A)J_1(A)}{J_k(A)N(A)}-
\sum_{N(A)> x^\frac{1}{k}}
\frac{\mu_1(A)J_1(A)}{J_k(A)N(A)} \\
= \sum_{A \in I_F}
\frac{\mu_1(A)J_1(A)}{J_k(A)N(A)}
+O\left( \sum_{N(A)> x^\frac{1}{k}}\frac{1}{N(A)^k} \right),
\end{align*}
the first term in (\ref{daiji2}) is 
\begin{align*}
\frac{c_F}{\zeta_F(k)}x \sum_{A \in I_F}
\frac{\mu_1(A)J_1(A)}{J_k(A)N(A)}
+ O \left( x^\frac{1}{k} \right). 
\end{align*}
Since 
\begin{align*}
\sum_{N(A) \leq \sqrt[k]{x}}
\frac{\sigma_{-\alpha}(A)}{N(A)}
&= 
\sum_{N(A) \leq \sqrt[k]{x}}N(A)^{-1}
\sum_{D|A}N(D)^{-\alpha} \\
&= 
\sum_{N(C) \leq \sqrt[k]{x}}N(C)^{-1}
\sum_{N(D) \leq \sqrt[k]{x}/N(C)}N(D)^{-\alpha-1} \\ 
&\ll  \sum_{N(C) \leq \sqrt[k]{x}}N(C)^{-1}
\ll \log x^{\frac{1}{k}},
\end{align*}
the O-term in (\ref{daiji2}) is $\ll x^\frac{1}{k}\log x$. 
Therefore, our assertion is proved.

\section{Proof of Theorem \ref{liouville}} 
According to the custom, we write $s=\sigma + i t$ for a complex number $s$. 
To prove Theorem \ref{liouville} we use the following lemma which is a special case of 
the classical Perron's formula (see e.g. \cite{titchmarsh}, p.70 ). 

\begin{lem}\label{perron}
Let   
\begin{align*}
f(s) =\sum_{n=1}^{\infty} \frac{a_n}{n^s}, \;\;\;\;\;\;  
B(\sigma)=\sum_{n=1}^{\infty} \frac{|a_n|}{n^{\sigma}},
\end{align*}
$\psi(x)>0$ non-decreasing function such that $a_n \ll \psi (n)$, 
and $\sigma_a$ the abscissa of absolute convergence of $f(s)$. 
 Then, for $b> \sigma_a$, 
\begin{align*}
\sum_{n \leq x}a_n=\frac{1}{2\pi i}\int_{b-iT}^{b+iT} f(s)\frac{x^s}{s}ds
+ O\left( \frac{\psi(2x)x\log x}{T} \right) \\
+ O\left( \frac{x^bB(b)}{T} \right)
+ O \left( \psi(N) \min\left[ \frac{x}{T|x-N|}, 1 \right]  \right)
\end{align*}
where $N$ is the integer nearest to $x$ and $T$ is a sufficiently large positive number. 
\end{lem}

When applying Lemma \ref{perron} 
to the Riemann zeta function, bounds for $a_n$ have no problem. 
However, 
to show our first assertion of Theorem \ref{liouville} in the case $F \neq \mathbb{Q}$, 
we need the following which is a new type of Perron's formula du to Liu and Ye \cite{liuye} as long as the author knows. 

\begin{lem}
With the same notation as in Lemma \ref{perron}, we have 
\begin{align*}
\sum_{n \leq x}a_n=\frac{1}{2\pi i}\int_{b-iT}^{b+iT} f(s)\frac{x^s}{s}ds
+O \left( \sum_{x-x/\sqrt{T}< n \leq x+x/\sqrt{T}}|a_n| \right) 
+ O \left( \frac{x^bB(b)}{\sqrt{T}} \right)
\end{align*}
\end{lem}

Now, we shall show the first assertion. 
For simplicity, we assume $F=\mathbb{Q}$. 
Using Lemma \ref{hodai2} and Perron's formula,  for large positive numbers $x$ and $T$,
\begin{align*}
L_k(x) = \frac{1}{2\pi i}
\int_{1+\frac{1}{\log x}-iT}^{1+\frac{1}{\log x}+iT}
\frac{\zeta((k+1)s)}{\zeta(s)}\frac{x^s}{s}ds + O\left(\frac{x\log x}{T}\right)
\end{align*}
where $\zeta(s)$ is the ordinary Riemann zeta function. 
It is well-known that  
$\zeta(s)$ has no zero for $\sigma \geq 1-c/\log(|t|+4)$. 
Moreover, we have 
\begin{align*}
\frac{1}{\zeta(s)} \ll \log|t|
\end{align*} 
for $\sigma \geq 1-c/\log(|t|+4)$ and $|t| \geq t_0$ for some $t_0>0$ (\cite{montgomery}, Theorem 6.7). 
Thus, there are positive constants $c$ and $t_0$ such that 
\begin{align*}
\frac{\zeta((k+1)s)}{\zeta(s)} \ll \zeta((k+1)(1-c/\log(|t|+4)))\log|t|
\end{align*} 
for $\sigma \geq 1-c/\log(|t|+4)$ and $|t| \geq t_0$. 
Therefore, by Cauthy's theorem, 
\begin{align*}
&\int_{1+\frac{1}{\log x}-iT}^{1+\frac{1}{\log x}-iT}
\frac{\zeta((k+1)s)}{\zeta(s)}\frac{x^s}{s}ds \\
&=
\left( \int_{1+\frac{1}{\log x}-iT}^{1-\frac{c}{\log T}-iT}+ 
\int_{1-\frac{c}{\log T}-iT}^{1-\frac{c}{\log T}+iT}+ 
\int_{1-\frac{c}{\log T}+iT}^{1+\frac{1}{\log x}+iT} 
\right)
\frac{\zeta((k+1)s)}{\zeta(s)}\frac{x^s}{s}ds \\
& := S_1 +S_2 +S_3 .
\end{align*}
We may have 
\begin{align*}
S_1
\ll \beta_{k, T}\log T \int_{1+\frac{1}{\log x}}^{1-\frac{c}{\log T}}
\frac{x^\sigma}{\sigma-iT}d\sigma 
\ll \beta_{k, T}\frac{\log T}{T}x, 
\end{align*}
and 
\begin{align*}
S_3
\ll \beta_{k, T}\log T \int_{1-\frac{c}{\log T}}^{1+\frac{1}{\log x}} 
\frac{x^\sigma}{\sigma+iT}d\sigma 
\ll \beta_{k, T}\frac{\log T}{T}x
\end{align*}
where 
\begin{align*}
\beta_{k, T}=\zeta((k+1)(1-c/\log T)).
\end{align*}

The second integral is 
\begin{align*}
S_2 &\ll 
\left( 
\int_{-T}^{-t_0}+ 
\int_{-t_0}^{t_0}+ 
\int_{t_0}^{T}
\right)
\frac
{\zeta \left((k+1)(1-\frac{c}{\log T}+it)\right)}
{\zeta (1-\frac{c}{\log T}+it)}
\frac
{x^{1-\frac{c}{\log T}+it}} 
{1-\frac{c}{\log T}+it}
dt 
\\  &\ll \beta_{k, T}\log T \int_{-T}^{-t_0} \frac{x^{1-\frac{c}{\log T}}}{t}dt 
+ \beta_{k, T}\log T \int_{t_0}^{T} \frac{x^{1-\frac{c}{\log T}}}{t} dt \\ 
&\ll \beta_{k, T} (\log T)^2 x^{1-\frac{c}{\log T}}.
\end{align*}
Therefore, 
\begin{align*}
L_k(x)=O\left( \beta_{k, T}\frac{x\log T}{T}\right) 
+ O\left( \beta_{k, T}(\log T)^2 x^{1-\frac{c}{\log T}} \right) +
O \left( \frac{x \log x}{T} \right)
.
\end{align*}
We can take $T=\exp(\sqrt{\log x})$ and  obtain 
\begin{align*}
L_k(x) \ll  \zeta((k+1)(1-c/\sqrt{\log x} )) x \exp(-c'\sqrt{\log x}) .
\end{align*} 
Thus, the first assertion is proved. 

Note that if we use a better result on the zero-free region of the Riemann zeta 
function, we get a better estimate. For example, by Lemma 12.3 in Ivic \cite{ivic}, 
we have 
\begin{align*}
L_k(x) \ll \zeta((k+1)(1-c/\frac{(\log x)^{\frac{3}{5}}}
{(\log \log x)^{\frac{1}{5}}}  )) x \exp \left(-c'\frac{(\log x)^{\frac{3}{5}}}
{(\log \log x)^{\frac{1}{5}}} 
\right).
\end{align*}

In the case $F \neq \mathbb{Q}$,  for possibility of the existence of Siegel zero of 
$\zeta_F(s)$,  the O-constant in the first assertion is not 
effective.

To end this section, we shall show the second assertion of Theorem \ref{liouville} 
which is an equivalent theorem of the Riemann hypothesis 
(RH for short). 
First, assume that $L_k(x)=O(x^{\frac{1}{2}+\varepsilon})$. 
Since  
\begin{align*}
\sum_{n \leq x}\frac{\lambda_k(n)}{n^\sigma} &= 
\frac{1}{x^{\sigma}}L_k(x)+\sigma\int_{1}^{x}\frac{L_k(u)}{u^{\sigma+1}}du \\
&= O(x^{\frac{1}{2}+\varepsilon-\sigma})
\end{align*}
by using the partial summation formula, 
$\zeta \left( (k+1)s \right)/\zeta (s)$ is analytic for $\sigma >1/2$. 
Thus, $\zeta(s)$ has no zero in this range, that is,  RH holds.  

Next, assume  RH. 
By Perron's formula, 
\begin{align*}
L_k(x)=\frac{1}{2\pi i}\int_{2-iT}^{2+iT}
\frac{\zeta \left( (k+1)s \right)}{\zeta(s)}\frac{x^s}{s}ds +
 O\left(\frac{x^2}{T}\right) 
\end{align*}
for large positive numbers $x$ and $T$. 
If RH holds, there exists a positive number $t_0$ such that 
\begin{align*}
 \frac{1}{\zeta(s)}  = O(|t|^\varepsilon)
\end{align*}
for $|t| > t_0$ and $\sigma \geq \frac{1}{2}$ 
from Theorem 14.2 in \cite{titchmarsh}, 
p.336.  
So 
\begin{align*}
 \frac{\zeta \left( (k+1)s \right)}{\zeta(s)} = O(|t|^\varepsilon)
\end{align*}
in $|t|>t_0$ and $\sigma > 1/2 + \delta$ for any $\delta>0$. 
By this estimate, 
we have 
\begin{align*}
L_k(x) &=\frac{1}{2\pi i}
\left( \int_{2-iT}^{\frac{1}{2}+\delta-iT}
+\int_{\frac{1}{2}+\delta-iT}^{\frac{1}{2}+\delta+iT}
+\int_{\frac{1}{2}+\delta+iT}^{2+iT} \right)
\frac{\zeta \left( (k+1)s \right)}{\zeta(s)}\frac{x^s}{s}ds +
 O\left(\frac{x^2}{T}\right) \\
&= O \left( x^2T^{\varepsilon-1}\right) 
+ O \left( x^{\frac{1}{2}+\delta}T^{\varepsilon} \right) 
+O \left( \frac{x^2}{T} \right) 
 \end{align*}
 for any small number $\delta>0$. Choosing $T=x^2$, we have 
 \begin{align*}
L_k(x)=O \left( x^{2\varepsilon}\right) 
+ O \left( x^{2\varepsilon+\frac{1}{2}+\delta} \right).   
\end{align*}
Since $\varepsilon$ and $\delta$ are arbitrary small number, 
the assertion holds.

Note that we may give another proof by using results on the ordinary Mobius function.
For a positive integer $k$, we have 
\begin{align*}
\lambda_k(A)=\sum_{D^{k+1}|A}\mu_1 \left( \frac{A}{D^{k+1}} \right).
\end{align*} 
(In the case $A=P^m$ where $P$ is a prime number and $m$ is a positive integer, 
we can easily see this. 
Since the both sides are multiplicative, the above equation holds.) 
We can also show our result using this formula and 
$\sum_{N(A) \leq x}\mu_1(x)=O(x^{\frac{1}{2}+\varepsilon})$ for any $\varepsilon>0$ 
if GRH holds (see \cite{titchmarsh}, p.370 or Theorem 1.3 of \cite{minamide}).

\section{Proof of Theorem \ref{qk}} 
In order to show Theorem \ref{qk}, 
we use the following proposition. 

\begin{prop}\label{k-freeprop} 
For a integer $k \geq 2$, we have 
\begin{align*}
Q_k(x)=\sum_{N(D) \leq \sqrt[k]{x}}\mu_1(D)
\left[ \frac{x}{N(D)^k} \right]_F.
\end{align*}
\end{prop}
\begin{proof}
By Lemma \ref{k-freelem}, 
\begin{align*}
Q_k(x) &= 
\sum_{N(A) \leq x}\sum_{D^{k}| A}\mu_1(D) \\
&= \sum_{N(D)^k \leq x}\sum_{N(B)\leq x/N(D)^k}\mu_1(D) 
= \sum_{N(D)^k \leq x}\mu_1(D)\left[  \frac{x}{N(D)^k} \right]_F.
 \end{align*}
\end{proof}

Let $R(x)=[x]_F-c_Fx$. 
By Proposition \ref{k-freeprop}, we have 
\begin{align*}
Q_k(x) &=\sum_{N(D) \leq \sqrt[k]{x}}\mu_1(D)
\left[ \frac{x}{N(D)^k} \right]_F \\
&=c_Fx\sum_{N(D) \leq \sqrt[k]{x}}\frac{\mu_1(D)}{N(D)^k}
+ \sum_{N(D) \leq \sqrt[k]{x}}\mu_1(D)R\left(\frac{x}{N(D)^k}\right) \\
&=\frac{c_Fx}{\zeta_F(k)} - c_Fx\sum_{N(D) > \sqrt[k]{x}}\frac{\mu_1(D)}{N(D)^k} + 
\sum_{N(D) \leq \sqrt[k]{x}}\mu_1(D)R\left(\frac{x}{N(D)^k}\right). 
\end{align*}
By the partial summation formula, we have 
\begin{align*}
\sum_{N(D) > \sqrt[k]{x}}\frac{\mu_1(D)}{N(D)^k} =O\left( x^{\frac{1}{k}-1}\right). 
\end{align*}
Thus, the second term is $O(x^{\frac{1}{k}})$. 

Since $R(x)=O(x^{1-\frac{2}{d+1}})$, the third term is  
\begin{align*}
\sum_{N(D) \leq \sqrt[k]{x}}\mu_1(D)R\left(\frac{x}{N(D)^k}\right)
&\ll 
 x^{1-\frac{2}{d+1}}
\sum_{N(D) \leq \sqrt[k]{x}}\frac{1}{N(D)^{k(1-\frac{2}{d+2})} }. 
\end{align*} 
We can check that 
\begin{align*}
 x^{1-\frac{2}{d+1}}
\sum_{N(D) \leq \sqrt[k]{x}}\frac{1}{N(D)^{k(1-\frac{2}{d+1})} }
=
\begin{cases}
O(x^{1-\frac{2}{d+1}}) \;\;\; \textrm{if $k(1-\frac{2}{d+1}) >1$,} \\
O(x^{\frac{1}{2}}) \;\;\; \textrm{if $(d, k)=(2, 2)$,} \\
O(x^{\frac{1}{3}}\log x) \;\;\; \textrm{if $(d, k)=(2, 3)$,} \\
O(x^{\frac{1}{2}}\log x) \;\;\; \textrm{if $(d, k)=(3, 2)$.}
\end{cases}
\end{align*}
Therefore, we have the assertion.

%In the case $F=\mathbb{Q}$ and $k \geq 2$, Walfisz \cite{Walfisz} showed 
%\begin{align*}
%Q_k(x)=x/\zeta(k)+O\left( \sqrt[k]{x} \exp \left(-c\frac{(\log x)^{\frac{3}{5}}}
%{k^{8}{5}(\log \log x)^{\frac{1}{5}}} 
%\right)\right)
%\end{align*}

%it is conjectured that 
%$Q_k(x)=x/\zeta(k)+O(x^{1/2k+\varepsilon})$. 

\quad

\quad

\noindent
Yusuke Fujisawa \\
Graduate School of Mathematics \\
Nagoya University \\
Chikusa-ku, Nagoya 464-8602 \\
Japan \\
%Mail: fujisawa.gifu@gmail.com 

\newpage


\begin{thebibliography}{9}



\bibitem{apostol}
T. M. Apostol, 
{\it M{\"o}bius functions of order $k$}, 
Pacific J. Math. {\bf 32} (1970), 21--27.

\bibitem{berndt} B. C. Berndt, 
{\it On the average order of ideal functions and other arithmetic functions}, 
Bull. Amer. Math. Soc. {\bf 76} (1970), 1270--1274.

\bibitem{cohen}
E. Cohen, 
{\it An elementary estimate for the $k$-free integers},
Bull. Amer. Math. Soc. {\bf 69} (1963), 762--765.

\bibitem{ivic}
A. Ivic, 
{\it The Riemann zeta function, theory and applications}, 
Dover Publications Inc, 2003.


\bibitem{liuye}
J. Liu and Y. Ye, 
{\it Perron's formula and the prime number theorem for automorphic L-functions}, 
Pure and Applied Math. Quart. {\bf 3} (2007), 481--497.



\bibitem{minamide}
Y. Fujisawa and M. Minamide, 
{\it On partial sums of the M{\"o}bius and Liouville functions for number fields}, 
preprint; arXiv math. 1212.4348.

\bibitem{murty}
 R. Murty and J. V. Order, 
 {\it Counting integral ideals in a number field},
Expo. Math. {\bf 25} (2007), 53--66. 







\bibitem{montgomery}
H. L. Montgomery and R. C. Vaughan, 
{\it Multiplicative Number Theory I. Classical Theory}, 
Cambridge University Press, 2007.





\bibitem{titchmarsh}
E. C. Titchmarsh, 
{\it The theory of the Riemann zeta-function},
2nd edition, revised by D. R. Heath-Brown, Oxford University Press, 1986.


\bibitem{walfisz} 
A. Walfisz, 
{\it Weylsche Exponentialsummen in der neueren Zahlentheorie}, (in German)
Mathematische Forschungsberichte, XV. VEB Deutscher Verlag der Wissenschaften, Berlin,  1963.


\end{thebibliography}
\end{document}